\documentclass{amsart}

\newtheorem{theorem}{Theorem}[section]
\newtheorem{lemma}[theorem]{Lemma}

\theoremstyle{definition}
\newtheorem{definition}[theorem]{Definition}
\newtheorem{example}[theorem]{Example}

\newtheorem{proposition}[theorem]{Proposition}
\newtheorem{corollary}[theorem]{Corollary}

\theoremstyle{remark}
\newtheorem{remark}[theorem]{Remark}

\numberwithin{equation}{section}
\begin{document}

\newcommand{\spacing}[1]{\renewcommand{\baselinestretch}{#1}\large\normalsize}
\spacing{1.14}

\title[Reilly formula for a class of elliptic differential operator]{Extension of Reilly formula for a class of elliptic differential operator in divergence form}

\author {Seyed Hamed Fatemi }

\address{Department of Mathematics, Tarbiat Modares University, Tehran, Iran.} \email{fatemi.shamed@gmail.com}
\author{Shahroud  Azami}
\address{Department of Pure Mathematics, Faculty of Sciences, Imam Khomeini International University, Qazvin, Iran. }
 \email{azami@sci.ikiu.ac.ir}

\date{\today}

\begin{abstract} In this paper, we prove the Reilly formula for the elliptic divergence type operator ${L_A}(u): = div(A\nabla u)$ on a compact Riemannian manifold $M$ where $A $ is a positive definite divergence free self-adjoint $(1,1)$-Codazzi tensor field on $M$ and then by assumption on extension of Ricci tensor we get some lower estimates for the first eigenvalue of  ${L_A}$.
\end{abstract}

\maketitle
\section{Introduction}
Elliptic operators on manifolds is one of the important extensions of the Laplace operator. One knows that a second-order linear differential operator without zero order term   $L:{C^\infty }(M) \to {C^\infty }(M)$ can be written as
\[Lf = div(A\nabla f) + \left\langle {V,\nabla f} \right\rangle, \]
where $ A \in \Gamma (End(TM)) $ is a self-adjoint with respect to the metric $\left\langle {\,,\,} \right\rangle $.
So an operator of the form ${L_A}(f) = div(A\nabla f)$ is an important kind of elliptic operator.
 One of the important issue associated with the operator $ L_A $  is study of the spectrum of this operator when the manifold $ M $ is compact. In this regard \cite{alencar2015eigenvalue,do2010inequalities,gomes2016eigenvalue,shi2016eigenvalue} got valuable results. One way to get a lower estimate of the first eigenvalue of the Laplace operator is Reilly formula \cite{reilly1977applications}. The formula is proved by integration from the usual Bochner formula and states that for each smooth function $u$ on a manifold $M$ one has,
 \begin{eqnarray}\nonumber
&&\int_M {\left( {{{\left( {\Delta u} \right)}^2} - {{\left| {hess(u)} \right|}^2} - Ric\left( {\nabla u,\nabla u} \right)} \right)} dvo{l_g} \\\label{originalReilly}&&= \int_{\partial M} {\left( {(n - 1)Hu_n^2 + II({\nabla ^\partial }u,{\nabla ^\partial }u) + 2{u_n}{\Delta ^\partial }(u)} \right)dvo{l_{\bar g}}} ,
\end{eqnarray}
where ${u_n} = \frac{{\partial u}}{{\partial \overrightarrow{n}}}$ and ${\nabla ^\partial },\,{\Delta ^\partial }$ are the gradient and Laplacian with respect to the metric of $\partial M$.
This  formula has many interesting consequences in geometry such as estimates of the first eigenvalue of the Laplace operator, Alexandrov's theorem and the Heintze-Karcher's inequality.
The Reilly formula (\ref{originalReilly}) is similarly proved for weighted manifolds \cite{ma2009extension} as follows,
\begin{eqnarray}\nonumber
  && \int_M \left( {{\left( {{\Delta _f}u} \right)}^2} - {{\left| {hess(u)} \right|}^2} - Ri{c_f}\left( {\nabla u,\nabla u} \right) \right)dvo{l_g}\\\label{weightedReilly}
&&
=   {\int_{\partial M} {u_n^{}\left( {Hu_n^{} - \left\langle {\nabla H,\nabla u} \right\rangle  + {\Delta ^\partial }(u)} \right)dvo{l_{\bar g}}} }  \\\nonumber
  && { + \int_{\partial M} {\left( {II({\nabla ^\partial }u,{\nabla ^\partial }u) - \left\langle {{\nabla ^\partial }u,{\nabla ^\partial }{u_n}} \right\rangle } \right)dvo{l_{\bar g}}} }.
\end{eqnarray}
Estimate of the first eigenvalue of the Laplace operator, is one of the important and long-standing problem in geometric analysis and PDE theory on manifolds. For example, it gives an upper bound for the constant in the Poincar\'{e} inequality. So it is very important to find a good lower estimate for the first eigenvalue of the Laplace operator. Another application  is in the estimate of the heat kernel \cite{chavel1984eigenvalues, chen1997poincare, li1980estimates,li2012geometric, ma2009extension}. Similar results have been obtained for weighted manifolds \cite{grigor2006heat}.\par
In this  paper, we get a Reilly-type formula for the elliptic divergence type operator\linebreak ${L_A}u = div(A\nabla u)$, when $A$ is a divergence free positive definite self-adjoint $(1,1)$-Codazzi tensor field on $M$  and obtain some lower estimates for the first eigenvalue of this operator. Also the approach of this paper is more similar to the corresponding results for the Laplace operator.\par
Explicitly, the results are as follows.
At first we get the Reilly formula for the elliptic operator ${L_A}$, when $ A $ is a parallel tensor field. As an important consequence we get the following estimates of the lower bound of the first eigenvalue of the operator ${L_A}$ , when $ A $ is parallel.
\begin{theorem}\label{eigen1}
Let $M$ be a closed  Riemannian manifold and $A$ be a parallel symmetric and positive semi-definite operator on $M$ such that one of the following conditions holds,
\begin{itemize}
\item[1)]$Ri{c_A} \ge Kg$ and $K > 0$ is a constant,
\item[2)]$Ri{c_A}(X,X) \ge K\left\langle {AX,X} \right\rangle$ and $K > 0$ is a constant.
\end{itemize}
Then, one has the following estimates for the first eigenvalue of the operator ${L_A}$,
\begin{itemize}
\item[1)]\begin{equation}
\lambda  \ge \frac{{Trace(A)K}}{{Trace(A) - {\delta _1}}},\label{e1b1}
\end{equation}
\item[2)]\begin{equation}
\lambda  \ge \frac{{Trace(A){\delta _1}K}}{{Trace(A) - {\delta _1}}},\label{e1b2}
\end{equation}
\end{itemize}
where $\lambda $ is the first positive eigenvalue of the operator ${L_A}$. If the equality holds then $ A $ is scalar operator, i.e. $A = \alpha I$ for some real constant $\alpha$ and $ M $ has constant sectional curvature $\frac{K}{{n - 1}}$.
\end{theorem}
As similar as the original one for the Laplace operator the estimates of the Theorem \ref{eigen1} are trivial when $ K \le 0 $, so by adapting of the Li and Yau method we get the following results when $ K \le 0 $.
\begin{theorem}\label{eigenpp2}
Let $ M $ be a closed Riemannian manifold, $A$ be a parallel symmetric and positive semi-definite operator on $M$ and $Ri{c_A} \ge  - K$ for some $K > 0$, then we have the following estimate for the first eigenvalue of the operator ${L_A}$,
\[2\left( {\alpha  + \sqrt {{\alpha ^2} + K\alpha } } \right)\exp \left( { - 1 - \sqrt {1 + \frac{K}{\alpha }} } \right) \le \lambda ,\]
where
$\alpha  = \frac{{\delta _1^2}}{{{d^2}Trace\left( A \right)}}$, $d=diam(M)$ and $ {\delta _1} $ is defined in Definition \ref{minmax}.
\end{theorem}
For the Codazzi divergence free tensor fields, we get the following extended Reilly formula,
\begin{theorem}[\textbf{Extended Reilly formula}]\label{Reilly}
Let $M$ be a Riemannian manifold with boundary $\partial M$ and $A$ be a $(1,1)$-Codazzi tensor field with $div\left( A \right) = 0$ then,
\[B = C,\]
where
\begin{eqnarray}\nonumber
B&=&\int_{\partial M}\left(\langle \nabla _{{\nabla ^\partial }u}{\nabla ^\partial }u,A\overrightarrow{n}\rangle - 2\langle sha{p^\partial }({\nabla ^\partial }u),A\overrightarrow{n}\rangle dvo{l_{\bar g}}  \right)\\&&+ \frac{1}{2}\int_{\partial M}\langle{\nabla ^\partial }u,\left( {\nabla _{\overrightarrow{n}}}A\right){\nabla ^\partial }u\rangle dvo{l_{\bar g}}
+ \int_{\partial M} \left( {u_n^2H_A^\partial } - \langle A\overrightarrow{n},{\nabla ^\partial }{u_n} \rangle  \right)dvo{l_{\bar g}}\\\nonumber&&
+ \frac{1}{2}\int_{\partial M} {u_n^2}\langle \overrightarrow{n},\left( {\nabla _{\overrightarrow{n}}}A \right)\overrightarrow{n}\rangle dvo{l_{\bar g}}
 + \int_{\partial M} \left( \left( {\nabla ^\partial }u.{u_n} \right)\langle {\overrightarrow{n},A\overrightarrow{n}} \rangle  - {u_n}{\Delta _A^\partial (u)} \right)dvo{l_{\bar g}} \\\nonumber&&+ \int_{\partial M} {u_n}\langle {\nabla ^\partial }u,\left( {\nabla _{\overrightarrow{n}}}A \right)\overrightarrow{n} \rangle dvo{l_{\bar g}}.
\end{eqnarray}
and
\begin{eqnarray}\nonumber
  C&=& \int_M\left( Trace\left( A \circ hes{s^2}(u) \right) \right) dvo{l_g} - \int_M {\Delta _A}(u)\left( \Delta u \right) dvo{l_g}\\&& + \int_M Ri{c_A}\left( \nabla u,\nabla u \right) dvo{l_g}
   + \frac{1}{2}\int_M \left\langle \nabla u,\left( \Delta A \right)\nabla u\right\rangle dvo{l_g}.
\end{eqnarray}
wherein $H_A^\partial : = Trace(A \circ shap{e^\partial })$ is defined as an extended mean curvature of the boundary $\partial M$, $\overrightarrow{n}$ is the outward unit vector field on $\partial M$ and $\overline g $ is the restricted metric on $\partial M$ and  ${\nabla ^\partial }$, $\Delta _A^\partial $ are gradient and extended Laplacian with respect to the metric ${\bar g}$.
\end{theorem}
Similarly we get the following estimates, for the first eigenvalues.
\begin{theorem}\label{eigennp1} Let $M$ be a closed Riemannian manifold and $A$ be a positive semi-definite $(1,1)$-Codazzi tensor on $M$ such that $ Trace(A) $ is constant. Also for each vector field $ X $ with $\left| X \right| = 1$ one has
\[Ri{c_A}\left( {X,X} \right) + Ric(X,AX) \ge 2K > 0,\]
then the following estimate for the first eigenvalue of ${L_A}$ is obtained,
\[\lambda  \ge \frac{{Trace\left( {A\,} \right)K}}{{Trace\left( {A\,} \right) - {\delta _1}}}.\]
\end{theorem}
Also, when $K \le 0$, we have the following result.
\begin{theorem}\label{eigennp2} Let $M$ be a closed  Riemannian manifold and $A$ be a positive semi-definite $(1,1)$-Codazzi tensor on $M$ such that $ Trace(A) $ is constant. Also, let  for each vector field $ X $ with $\left| X \right| = 1$ one has
\[Ric(X,AX) \ge  - K,\,\,\,\,\,{\delta} =\max \left\langle {X,(\nabla A )X} \right\rangle\]
and for any vector field $ X,Y ,Z$ we have
\[\left| {\left( {{\nabla _X}{\nabla _Y}A} \right)Z} \right| \le K'\left| X \right|\left| Y \right|\left| Z \right|\]
where $K, K' ,\delta\ge  0$ and $diam(M) \le d$.
Then the following estimate for the first eigenvalue of ${L_A}$ is obtained,
\[2\left( {\alpha  + \sqrt {{\alpha ^2} + (K+2K'+\delta)\alpha } } \right)\exp \left( { - 1 - \sqrt {1 + \frac{{K + 2K'+\delta}}{\alpha }} } \right) \le \lambda ,\]
where
$\alpha  = \frac{{\delta _1^2}}{{{d^2}Trace\left( A \right)}}.$
\end{theorem}
Finally, for more general case, we have the following result, when $K > 0$.
\begin{theorem}\label{more general}Let $ B $ be a $(1,1)$-self-adjoint tensor field and satisfies in following conditions
\begin{itemize}
\item[a)]$div(B) = 0$,

\item[b)]${\nabla ^2}B(X,Y)Z \le K'\left| X \right|\left| Y \right|\left| Z \right|$ for any vector field $ X,Y,Z $,

\item[c)]$Ric\left( {X,BX} \right) \ge K{\left| X \right|^2}$,

\item[d)]$\nabla Trace(B)$ is parallel.
\end{itemize}
Then the following estimate is obtained for the first eigenvalue of ${L_B}$,
\[\lambda  \ge \frac{{n{\delta _n}\left( {K + 2nK'} \right)}}{{n{\delta _n} - {\delta _1}}}.\]
\end{theorem}

\section{Preliminaries}
In this section, we summarize some preliminaries that we use in throughout paper.
\begin{definition} A $(1,1)$-tensor field $ A $ on a Riemannian manifold $M,\langle,\rangle)$ is  self-adjoint whenever
\[\forall X,Y \in \Gamma \left( {TM} \right):\left\langle {AX,Y} \right\rangle  = \left\langle {X,AY} \right\rangle. \]
\end{definition}
\begin{definition}\label{minmax}
Let $  A $ be a self-adjoint positive definite (1,1)-tensor field on $ M $, we say $ A $ is bounded if there are some constant $\alpha ,\beta  > 0$ such that for any vector field $X \in \Gamma (TM)$ on $ M $ with $\left| X \right| = 1$, one has $\alpha  < \left\langle {X,AX} \right\rangle  < \beta$ and ${\delta _1},{\delta _n}$ are defined as follows
\begin{itemize}
\item[a)] ${\delta _1} = \mathop {\min }\limits_{\left| X \right| = 1} \left\langle {X,AX} \right\rangle, $
\item[b)] ${\delta _n} = \mathop {\max }\limits_{\left| X \right| = 1} \left\langle {X,AX} \right\rangle. $
\end{itemize}
Note that when $  A $ is parallel, then $ \left\langle {\nabla r,A\nabla r} \right\rangle $ is constant with respect to distant function \linebreak $r(x) = dist(p,x)$, in other words $\frac{\partial }{{\partial r}}\left\langle {\nabla r,A\nabla r} \right\rangle  = 0$. So $ {\delta _1} = \mathop {\min }\limits_{B\left( {p,\varepsilon } \right)} \left\langle {\nabla r,A\nabla r} \right\rangle  $ and \linebreak $ {\delta _n} = \mathop {\max }\limits_{B\left( {p,\varepsilon } \right)} \left\langle {\nabla r,A\nabla r} \right\rangle $ which $ \varepsilon>0$ is arbitrary.
\end{definition}
\begin{definition} \label{ricci}
Let $A $ be a self-adjoint operator on manifold $M$ and $\left\{ {{e_i}} \right\}$ be an orthonormal basis at the computing point. We define $L_{A}, \,\,\Delta_{A}$ and $Ric_{A}$ as follow
\begin{itemize}
\item[a)]$ {L_A}(u): = div\left( {A\nabla u} \right) = \sum\nolimits_i {\left\langle {{\nabla _{{e_i}}}\left( {A\nabla u} \right),{e_i}} \right\rangle }  $,
\item[b)]$ {\Delta _A}(u): = \sum\nolimits_i {\left\langle {{\nabla _{{e_i}}}\nabla u,A{e_i}} \right\rangle } $,
\item[c)]$Ri{c_A}(X,Y): = \sum\nolimits_i {\left\langle {R(X,A{e_i}){e_i},Y} \right\rangle }$ and we call the tensor $Ri{c_A}$ as an extended Ricci tensor,
\end{itemize}
where $u$ is a smooth function and $X,Y$ are vector fields on $M$.
\end{definition}
As usual for comparison results in differential geometry one needs a Bochner formula and the associated Riccati inequality. The following Theorem provided this.
\begin{theorem}[\textbf{Extended Bochner formula}]\cite{gomes2016eigenvalue}\label{extendedBochner}
Let $M$ be a smooth Riemannian manifold and $A$ be a self-adjoint operator on $M$, then for any smooth function $u$ on $M$, we have,
\begin{eqnarray}\nonumber
\frac{1}{2}{L_A}(\left| \nabla u \right|^2)&=&\frac{1}{2}\langle \nabla{\left| {\nabla u} \right|}^2 div(A)\rangle
+ Trace\left( A \circ hes{s^2}\left( u \right)\right) +\langle \nabla u,\nabla ({\Delta _A}u)\rangle\\\label{b0}&&
   -\Delta _{\left( \nabla _{\nabla u}A \right)}u + Ri{c_A}(\nabla u,\nabla u).
\end{eqnarray}
\end{theorem}
In the following proposition, we provide a generalization of Cauchy-Schwartz inequality. By this result, we can get the so-called Riccati inequality to the extended Bochner formula in Theorem \ref{extendedBochner}  in a similar way.
\begin{proposition}\cite{alencar2015eigenvalue}Let $ A $ be a positive semi-definite symmetric matrix, then for every matrix $ F $ we have,
\begin{equation}
Trace\left( {A{F^2}} \right) \ge \frac{1}{{Trace(A)}}{\left( {Trace(AF)} \right)^2}.
\label{newton}
\end{equation}
and the equality holds if and only if $A = \alpha I$ for some $\alpha  \in \mathbb{R}$.
\end{proposition}
\begin{definition}
Let $(M,\langle ,\rangle)$ be a Riemannian manifold and $A$ be a $ (1,1)$-tensor field. We say that the tensor $A$  is a Codazzi Tensor if
 $(\nabla _{X}A)\langle Y,Z\rangle=(\nabla _{Y}A)\langle X,Z\rangle$.
\end{definition}

\begin{definition} Let $ A $ be a $ (1,1)$-tensor field on manifold $ M $, then we define $T^{A}$ as follows
  \[{T^A}(X,Y): = \left( {{\nabla _X}A} \right)Y - \left( {{\nabla _Y}A} \right)X.\]
Notice,   $T$ is a $(2,1)$ tensor field.
\end{definition}
\begin{example}
If $A$ is the shape operator of the hypersurface ${\Sigma ^n} \subset {M^{n + 1}}$ then \[{T^A}(X,Y) = {\left( {\bar R(Y,X)N} \right)^T},\]
where ${\bar R}$ is the curvature tensor on the ambient manifold $ M $, and $ N $ is the unit normal vector field on ${\Sigma ^n} \subset {M^{n + 1}}$.
\end{example}
We compute the second covariant derivation of the operator $ A $. This lemma is useful for computation of the tensor $Ri{c_A}$ and its relation with other geometric quantities like Laplace of the tensor $ A $ and the Ricci tensor.
\begin{lemma}\label{laplaceA1} Let $ A $ be a self-adjoint operator on manifold $ M $ and $ X,Y,Z $  are vector fields on $ M $, then
\begin{itemize}
\item[a)]$\left( {{\nabla ^2}A} \right)\left( {X,Y,Z} \right) = \left( {{\nabla ^2}A} \right)\left( {X,Z,Y} \right) + R(Z,Y)\left( {AX} \right) - A\left( {R(Z,Y)X} \right),$
\item[b)]$\left( {{\nabla ^2}A} \right)\left( {X,Y,Z} \right) - \left( {{\nabla ^2}A} \right)\left( {Y,X,Z} \right) = \left( {{\nabla _Z}T^A} \right)(X,Y).$
\end{itemize}
\end{lemma}
\begin{proof}
For part (a) we have,
\begin{eqnarray*}
\nabla ^{2}A(X,Y,Z) &=&\left( \nabla \left(\nabla A \right) \right)(X,Y,Z) = \left( \nabla _{Z}\left( \nabla A \right) \right)(X,Y)  \\
  & =& \nabla_{Z}\left( \left( \nabla A \right)(X,Y) \right) - \left( \nabla A \right)(\nabla_{Z}X,Y) - \left( \nabla A \right)(X,\nabla_{Z}Y) \\
   &=& \nabla_{Z}\left(\left( \nabla_{Y}A \right)X \right) - \left( \nabla A \right)(\nabla_{Z}X,Y) - \left( \nabla_{\nabla_{Z}Y}A \right)(X)  \\
     &=& \left( \nabla_{Z}\left( \nabla_{Y}A \right) \right)X + \left( \nabla_{Y}A \right)\left( \nabla_{Z}X \right) - \left( \nabla_{Y}A \right)(\nabla_{Z}X) - \left( \nabla _{{\nabla _Z}Y}A \right)(X)\\
 &=&\left( {\nabla _Z}\left( {{\nabla _Y}A} \right) \right)X - \left( \nabla _{{\nabla _Z}Y}A \right)X.
\end{eqnarray*}
Similarly,
\[{\nabla ^2}A(X,Z,Y) = \left( {{\nabla _Y}\left( {{\nabla _Z}A} \right)} \right)X - \left( {{\nabla _{{\nabla _Y}Z}}A} \right)X.\]
 Thus
\[\begin{array}{*{20}{c}}
   {{\nabla ^2}A(X,Y,Z) - {\nabla ^2}A(X,Z,Y){\rm{ }}} \hfill & { = \left( {{\nabla _Z}{\nabla _Y}A} \right)X - \left( {{\nabla _Y}\left( {{\nabla _Z}A} \right)} \right)X - \left( {{\nabla _{\left[ {Z,Y} \right]}}A} \right)X} \hfill  \\
   {} \hfill & { = \left( {R(Z,Y)A} \right)X = R(Z,Y)\left( {AX} \right) - A\left( {\left( {R(Z,Y)X} \right)} \right).} \hfill  \\
\end{array}\]
For part (b), by definition of $ T $, we have
\begin{eqnarray*}
{\nabla ^2}A(X,Y,Z)&=& \left( {{\nabla _Z}\left( {\nabla A} \right)} \right)\left( {X,Y} \right)
\\&=& {\nabla _Z}\left( {\left( {\nabla A} \right)\left( {X,Y} \right)} \right) - \left( {\nabla A} \right)\left( {{\nabla _Z}X,Y} \right) - \left( {\nabla A} \right)\left( {X,{\nabla _Z}Y} \right)
\\&=& {\nabla _Z}\left( {\left( {\nabla A} \right)\left( {Y,X} \right) + {T^A}\left( {X,Y} \right)} \right) - \left( {\nabla A} \right)\left( {{\nabla _Z}X,Y} \right) - \left( {\nabla A} \right)\left( {X,{\nabla _Z}Y} \right)
\\&=&{\nabla _Z}\left( {\left( {\nabla A} \right)\left( {Y,X} \right)} \right) + {\nabla _Z}\left( {{T^A}\left( {X,Y} \right)} \right) - \left( {\nabla A} \right)\left( {{\nabla _Z}X,Y} \right)
 \\&&- \left( {\nabla A} \right)\left( {X,{\nabla _Z}Y} \right)
\\&=&\left( {{\nabla _Z}\left( {\nabla A} \right)\left( {Y,X} \right)} \right) + \left( {\nabla A} \right)\left( {{\nabla _Z}Y,X} \right) + \left( {\nabla A} \right)\left( {Y,{\nabla _Z}X} \right)
\\&&+ {\nabla _Z}\left( {{T^A}\left( {X,Y} \right)} \right) - \left( {\nabla A} \right)\left( {{\nabla _Z}X,Y} \right) - \left( {\nabla A} \right)\left( {X,{\nabla _Z}Y} \right)
\\&=& \left( {{\nabla _Z}\left( {\nabla A} \right)\left( {Y,X} \right)} \right) + {\nabla _Z}\left( {{T^A}\left( {X,Y} \right)} \right) + {T^A}\left( {{\nabla _Z}Y,X} \right) + {T^A}\left( {Y,{\nabla _Z}X} \right)
\\&=&\left( {{\nabla _Z}\left( {\nabla A} \right)\left( {Y,X} \right)} \right) + \left( {{\nabla _Z}{T^A}} \right)\left( {X,Y} \right).
\end{eqnarray*}
\end{proof}
\begin{lemma}\label{laplaceA10} Let $ A $ be a $ (1,1)-$symmetric tensor field and ${\nabla ^*}{T^A} = 0$, then
\[\left\langle {\left( {\Delta A} \right)X,X} \right\rangle  = \left\langle {{\nabla _X}div(A),X} \right\rangle  - Ri{c_A}(X,X) + Ric(X,AX),\]
where $\nabla ^*$ is adjoint of $\nabla$.
\end{lemma}
\begin{proof} For simplicity let $\left\{ {{e_i}} \right\}$ be an orthonormal local frame field with ${\nabla _{{e_i}}}{e_j} = 0$ at the computing point. By computation and Lemma \ref{laplaceA1} we have,
\begin{eqnarray*}
 \left\langle {\left( {\Delta A} \right)X,X} \right\rangle  &=& \sum\nolimits_i {\left\langle {\left( {{\nabla _{{e_i}}}{\nabla _{{e_i}}}A} \right)X,X} \right\rangle }  = \sum\nolimits_i {\left\langle {{\nabla ^2}A(X,{e_i},{e_i}),X} \right\rangle }
\\&=&\sum\nolimits_i {\left\langle {{\nabla ^2}A({e_i},X,{e_i}),X} \right\rangle }  + \sum\nolimits_i {\left\langle {\left( {{\nabla _{{e_i}}}{T^A}} \right)(X,{e_i}),X} \right\rangle }
\\&=&\sum\nolimits_i {\left\langle {{\nabla ^2}A({e_i},X,{e_i}),X} \right\rangle }  - \sum\nolimits_{i} \left\langle {div\left( {{T^A}} \right)X,X} \right\rangle
\\&=& \sum\nolimits_{i }\left\langle {{\nabla ^2}A({e_i},X,{e_i}),X} \right\rangle.
\end{eqnarray*}
 So by Lemma \ref{laplaceA1}, part (a) we have
\begin{eqnarray*}
\left\langle {\left( {\Delta A} \right)X,X} \right\rangle & =& \sum\nolimits_i {\left\langle {{\nabla ^2}A({e_i},X,{e_i}),X} \right\rangle }
\\&=& \sum\nolimits_{i} \left\langle {\nabla ^2}A({e_i},{e_i},X) + R({e_i},X)\left( {A{e_i}} \right) - A\left( \left( R({e_i},X){e_i} \right) \right),X \right\rangle
\\&=&\left\langle {\left( {{\nabla _X}divA} \right),X} \right\rangle  - Ri{c_A}\left( {X,X} \right) + Ric\left( {X,AX} \right).
\end{eqnarray*}
\end{proof}
\begin{corollary}\label{laplaceA} Let $ A $ be a self-adjoint $ (1,1)$-Codazzi tensor field. So $ T^A =0$ and we have,
\[\left\langle {\left( {\Delta A} \right)X,X} \right\rangle  = Hess(TraceA)(X,X) - Ri{c_A}(X,X) + Ric(X,AX).\]
Also, when ${\Sigma ^n} \subset {M^{n + 1}}$ be the hypersurface with shape operator $ A $ and the ambient manifold has constant sectional curvature $ c $, then
\[\left( {\Delta A} \right)X = n{\nabla _X}\nabla H - cnHX + \left( {cn - {{\left| A \right|}^2}} \right)AX + nH{A^2}X\]
\end{corollary}
\begin{proof} The first part is clear and for the second part refer to \cite{alias2016maximum} page 333.
\end{proof}
\begin{corollary}\label{exbfp}In Theorem \ref{extendedBochner} if $ A $ is parallel, then
\begin{equation}
\frac{1}{2}{L_A}({\left| {\nabla u} \right|^2}) = Trace\left( {A{\mkern 1mu} o{\mkern 1mu} hes{s^2}\left( u \right)} \right) + \left\langle {\nabla u,\nabla ({\Delta _A}u)} \right\rangle  + Ri{c_A}(\nabla u,\nabla u)
\label{b0p}
\end{equation}
\end{corollary}
The extended Bochner formula (\ref{b0}) is very complicated. The complication of the formula is about the existence of parameter ${\Delta _{{\nabla _u}A}}u $. However, when ${{\Delta _{{\nabla _u}A}}u \le 0}$, the Bochner formula get the simple Riccati inequality. In the following proposition, we get some parameters which seems are suitable to estimate ${{\Delta _{{\nabla _u}A}}u}$. These parameters show the affection of parallelness and values of eigenfunction on the value of ${{\Delta _{{\nabla _u}A}}u}$. The values of these parameter depend on analytic and algebraic properties of the tensor $ B $.
\begin{proposition}\label{lcomplicated} Let $ B $ be a $(1,1)$-self-adjoint tensor field on the manifold $ M $ then
\begin{eqnarray}\nonumber
\Delta_{\left( \nabla _{\nabla_{u}}B \right)}u&=&\nabla u.\nabla u.Trace(B) - \left\langle \nabla u,\left( \Delta B\right)\nabla u \right\rangle  + \sum\nolimits_{i} \left\langle T^{\left( \nabla_{\nabla u}B \right)}(\nabla u,{e_i}),{e_i} \right\rangle
\\&&+\sum\nolimits_{i} {e_i}.\left\langle \nabla u,{T^B}({e_i},\nabla u) \right\rangle .
\end{eqnarray}
\end{proposition}
\begin{proof}
Let $ B $ be a  $(1,1)$-tensor field, then
\begin{eqnarray*}
\Delta _{\left( \nabla _{\nabla_{u}}B \right)}u&=&\sum\nolimits_i {\left\langle {{\nabla _{{e_i}}}\nabla u,\left( {{\nabla _{{\nabla _u}}}B} \right){e_i}} \right\rangle }
\\&=&\sum\nolimits_i {\left\langle {{\nabla _{{e_i}}}\nabla u,\left( {{\nabla _{{e_i}}}B} \right)\nabla u} \right\rangle }  + \sum\nolimits_i {\left\langle {{\nabla _{{e_i}}}\nabla u,{T^B}(\nabla u,{e_i})} \right\rangle }
\\&=& \sum\nolimits_i {{e_i}.\left\langle {\nabla u,\left( {{\nabla _{{e_i}}}B} \right)\nabla u} \right\rangle  - \sum\nolimits_i {\left\langle {\nabla u,\left( {\nabla _{{e_i}}^2B} \right)\nabla u} \right\rangle } }
\\&&-\sum\nolimits_i {\left\langle {{\nabla _{{e_i}}}\nabla u,\left( {{\nabla _{{e_i}}}B} \right)\nabla u} \right\rangle }  + \sum\nolimits_i {\left\langle {{\nabla _{{e_i}}}\nabla u,{T^B}(\nabla u,{e_i})} \right\rangle }
\\&=& \sum\nolimits_i {{e_i}.\left\langle {\nabla u,\left( {{\nabla _{{e_i}}}B} \right)\nabla u} \right\rangle }  - \left\langle {\nabla u,\left( {\Delta B} \right)\nabla u} \right\rangle  - {\Delta _{\left( {{\nabla _{{\nabla _u}}}B} \right)}}u
\\&&-\sum\nolimits_i {\left\langle {{\nabla _{{e_i}}}\nabla u,{T^B}({e_i},\nabla u)} \right\rangle }  + \sum\nolimits_i {\left\langle {{\nabla _{{e_i}}}\nabla u,{T^B}(\nabla u,{e_i})} \right\rangle }
\\&=&\sum\nolimits_i {{e_i}.\left\langle {\nabla u,\left( {{\nabla _{{e_i}}}B} \right)\nabla u} \right\rangle }  - \left\langle {\nabla u,\left( {\Delta B} \right)\nabla u} \right\rangle  - {\Delta _{\left( {{\nabla _{{\nabla _u}}}B} \right)}}u
\\&&+ 2\sum\nolimits_i {\left\langle {{\nabla _{{e_i}}}\nabla u,{T^B}(\nabla u,{e_i})} \right\rangle }
\\&=& \sum\nolimits_i {{e_i}.\left\langle {\nabla u,\left( {{\nabla _{\nabla u}}B} \right){e_i}} \right\rangle }  + \sum\nolimits_i {{e_i}.\left\langle {\nabla u,{T^B}({e_i},\nabla u)} \right\rangle }  - \left\langle {\nabla u,\left( {\Delta B} \right)\nabla u} \right\rangle
\\&&- {\Delta _{\left( {{\nabla _{{\nabla _u}}}B} \right)}}u + 2\sum\nolimits_i {\left\langle {{\nabla _{{e_i}}}\nabla u,{T^B}(\nabla u,{e_i})} \right\rangle }
\\&=& {\Delta _{\left( {{\nabla _{{\nabla _u}}}B} \right)}}u + \sum\nolimits_i {\left\langle {\nabla u,\left( {{\nabla _{{e_i}}}{\nabla _{\nabla u}}B} \right){e_i}} \right\rangle }  + \sum\nolimits_i {\left\langle {{\nabla _{{e_i}}}\nabla u,{T^B}({e_i},\nabla u)} \right\rangle }
\\&&+ \sum\nolimits_i {\left\langle {\nabla u,{\nabla _{{e_i}}}\left( {{T^B}({e_i},\nabla u)} \right)} \right\rangle }  - \left\langle {\nabla u,\left( {\Delta B} \right)\nabla u} \right\rangle  - {\Delta _{\left( {{\nabla _{{\nabla _u}}}B} \right)}}u
\\&&+ 2\sum\nolimits_i {\left\langle {{\nabla _{{e_i}}}\nabla u,{T^B}(\nabla u,{e_i})} \right\rangle }.
\end{eqnarray*}
In other words,
\[{\Delta _{\left( {{\nabla _{{\nabla _u}}}B} \right)}}u = \left\langle {\nabla u,div\left( {{\nabla _{\nabla u}}B} \right)} \right\rangle  + \sum\nolimits_i {{e_i}.\left\langle {\nabla u,{T^B}({e_i},\nabla u)} \right\rangle }  - \left\langle {\nabla u,\left( {\Delta B} \right)\nabla u} \right\rangle .\]
But,
\begin{eqnarray*}
\left\langle {\nabla u,div\left( {{\nabla _{\nabla u}}B} \right)} \right\rangle &=&\sum\nolimits_i {\left\langle {\nabla u,{\nabla _{{e_i}}}\left( {{\nabla _{\nabla u}}B} \right){e_i}} \right\rangle }  = \sum\nolimits_i {\left\langle {{\nabla _{{e_i}}}\left( {{\nabla _{\nabla u}}B} \right)\nabla u,{e_i}} \right\rangle }
\\&=&\sum\nolimits_i {\left\langle {{\nabla _{\nabla u}}\left( {{\nabla _{\nabla u}}B} \right){e_i} + {T^{\left( {{\nabla _{\nabla u}}B} \right)}}(\nabla u,{e_i}),{e_i}} \right\rangle }
\\&=& \nabla u.\nabla u.Trace(B) + \sum\nolimits_i {\left\langle {{T^{\left( {{\nabla _{\nabla u}}B} \right)}}(\nabla u,{e_i}),{e_i}} \right\rangle } .
\end{eqnarray*}
So,
\begin{eqnarray*}
{\Delta _{\left( {{\nabla _{{\nabla _u}}}B} \right)}}u&=& \nabla u.\nabla u.Trace(B) - \left\langle {\nabla u,\left( {\Delta B} \right)\nabla u} \right\rangle  + \sum\nolimits_i {\left\langle {{T^{\left( {{\nabla _{\nabla u}}B} \right)}}(\nabla u,{e_i}),{e_i}} \right\rangle }
\\&&+ \sum\nolimits_i {{e_i}.\left\langle {\nabla u,{T^B}({e_i},\nabla u)} \right\rangle .}
\end{eqnarray*}
\end{proof}

We recall the following result from \cite{daicomparison}, Corollary 1.7.5. It is used for the rigidity result in Theorem \ref{eigen1}.
\begin{proposition}\cite{daicomparison}\label{hess} Let $ M $ be a complete Riemannian manifold, then $ M $ has constant sectional curvature $ H $ iff there is a non trivial smooth function $ u $ on $ M $ with $Hessu =  - Hug$, where $ g $ is the Riemannian metric on $ M $.
\end{proposition}
\section{Reilly formula, when $ A $ is parallel}
In this section, we prove the extended Reilly formula \ref{Reilly} when $ A $ is parallel. In fact the result is valid when ${\Delta _{{\nabla _{\nabla u}}A}}u = 0$.  As usually we integrate from the extended Bochner formula \ref{b0p}. The computation is coordinate-independent.
\begin{theorem}(\textbf{Reilly-type formula 1})\label{Reillyp} Let $ M $ be a complete Riemannian manifold and $ A $ be a parallel self-adjoint
$(1,1)$-tensor field on it, then
\begin{equation}\label{RR1}
B = C,
\end{equation}
where
\begin{eqnarray*}
B&=& \int_{\partial M} {\left( {\left\langle {{\nabla _{{\nabla ^\partial }u}}{\nabla ^\partial }u,A\overrightarrow{n}} \right\rangle  - 2{u_n}\left\langle {sha{p^\partial }({\nabla ^\partial }u),A\overrightarrow{n}} \right\rangle } \right)dvo{l_{\bar g}}}
\\&&+\int_{\partial M} {\left( {u_n^2H_A^\partial  - {u_n}\left\langle {A\overrightarrow{n},{\nabla ^\partial }{u_n}} \right\rangle } \right)dvo{l_{\bar g}}}
\\&&+\int_{\partial M} {\left( {\left( {{\nabla ^\partial }u.{u_n}} \right)\left\langle {\overrightarrow{n},A\overrightarrow{n}} \right\rangle  - {u_n}L_A^\partial (u)} \right)dvo{l_{\bar g}},}
\end{eqnarray*}
and
\[C = \int_M {\left( {Trace\left( {A\circ hes{s^2}(u)} \right)} \right)} dvo{l_{ g }} - \int_M {{L _A}(u)\left( {\Delta u} \right)} dvo{l_{g }} + \int_M {Ri{c_A}\left( {\nabla u,\nabla u} \right)} dvo{l_{ g }},\]
wherein $H_A^\partial : = Trace(A \circ shap{e^\partial })$ is defined as an extended mean curvature of the boundary $\partial M$, $\overrightarrow{n}$ is the outward unit vector field on $\partial M$ and $\overline g $ is the restricted metric on $\partial M$ and  ${\nabla ^\partial }$, $L _A^\partial $ are gradient and extended Laplacian with respect to the metric ${\bar g}$, respectively.
\end{theorem}
\begin{proof}Since $A$ is parallel we have $L_{A}=\Delta_{A}$. Therefore
integration from the extended Bochner formula in corollary \ref{exbfp} gives,
\begin{eqnarray}\nonumber
 \frac{1}{2}\int_M {{L_A}({{\left| {\nabla u} \right|}^2})dvo{l_g}}  &=& \int_M {Trace\left( {A\circ  hes{s^2}\left( u \right)} \right)dvo{l_g}}  + \int_M {\left\langle {\nabla u,\nabla ({L _A}u)} \right\rangle dvo{l_g}}
\\\label{R1}&&+\int_M {Ri{c_A}(\nabla u,\nabla u)dvo{l_g}}.
\end{eqnarray}
Each of the terms in the above formula are computed as follows,
\begin{equation}\label{R2}
\int_M {{L_A}({{\left| {\nabla u} \right|}^2})dvo{l_g}}  = \int_{\partial M} {\left\langle {\nabla ({{\left| {\nabla u} \right|}^2}),A\overrightarrow{n}} \right\rangle dvo{l_{\bar g}}},
\end{equation}
and
\begin{equation}\label{R3}
\int_M {\left\langle {\nabla u,\nabla ({L _A}u)} \right\rangle dvo{l_g}}  =  - \int_M {({L _A}u)(\Delta u)dvo{l_g}}  + \int_{\partial M} {({L_A}u)\left\langle {\nabla u,\overrightarrow{n}} \right\rangle dvo{l_{\bar g}}}.
\end{equation}
So, from (\ref{R1}), (\ref{R2}) and (\ref{R3}) with re-arrangement we have,
\[B = C,\]
where
\begin{equation}
B = \frac{1}{2}\int_{\partial M} {\left\langle {\nabla ({{\left| {\nabla u} \right|}^2}),A\overrightarrow{n}} \right\rangle dvo{l_{\bar g}}}  - \int_{\partial M} {({L_A}u)\left\langle {\nabla u,\overrightarrow{n}} \right\rangle dvo{l_{\bar g}}},
\label{R4}
\end{equation}
and
\[C = \int_M {Trace\left( {A\circ hes{s^2}\left( u \right)} \right)dvo{l_g} - \int_M {({L_A}u)(\Delta u)dvo{l_g}} }  + \int_M {Ri{c_A}(\nabla u,\nabla u)dvo{l_g}}. \]
For Dirichlet or Neumann boundary conditions we should compute (\ref{R4}) on the boundary of $ M $, in other words, we need to compute the following parameters with respect to the intrinsic and extrinsic geometry of $\partial M$,
\begin{equation}
\left\langle {\nabla ({{\left| {\nabla u} \right|}^2}),A\overrightarrow{n}} \right\rangle,
\label{rr1}
\end{equation}
and
\begin{equation}
({L_A}u)\left\langle {\nabla u,\overrightarrow{n}} \right\rangle.
\label{rr2}
\end{equation}
 We denote the gradient and $A-$Laplacian with respect to the geometry of $\partial M$ by ${\nabla ^\partial }$ and ${L_A ^\partial }$ and the gradient and $A-$Laplacian with respect to $M$ by $\nabla $ and $L_A $. Now, we compute each item.
Computation of (\ref{rr1}) results that
\[\frac{1}{2}\left\langle {\nabla ({{\left| {\nabla u} \right|}^2}),A\overrightarrow{n}} \right\rangle  = \frac{1}{2}A\overrightarrow{n}.\left\langle {\nabla u,\nabla u} \right\rangle  = \left\langle {{\nabla _{A\overrightarrow{n}}}\nabla u,\nabla u} \right\rangle. \]
Now, for computation of (\ref{rr2}),
let $\left\{ {{e_i}} \right\}$ be local orthonormal frame field such that ${e_n} = \overrightarrow{n}$ be the outward unit vector field on $ \partial M $, then
\begin{eqnarray*}
 {{L_A}u}&=& \sum\nolimits_{i < n} {\left\langle {\nabla _{{e_i}}^{}\nabla u,A{e_i}} \right\rangle }  + \left\langle {{\nabla _{\overrightarrow{n}}}\nabla u,A\overrightarrow{n}} \right\rangle
\\&=& \sum\nolimits_{i < n} {\left\langle {\nabla _{{e_i}}^{}\left( {{\nabla ^\partial }u + \left( {{u_n}} \right)\overrightarrow{n}} \right),A{e_i}} \right\rangle }  + \left\langle {{\nabla _{\overrightarrow{n}}}\nabla u,A\overrightarrow{n}} \right\rangle
\\&=& \sum\nolimits_{i < n} {\left\langle {\nabla _{{e_i}}^{}{\nabla ^\partial }u,A{e_i}} \right\rangle }  + \sum\nolimits_{i < n} {\left\langle {\nabla _{{e_i}}^{}\left( {\left( {{u_n}} \right)\overrightarrow{n}} \right),A{e_i}} \right\rangle }  + \left\langle {{\nabla _{\overrightarrow{n}}}\nabla u,A\overrightarrow{n}} \right\rangle
\\&=&\sum\nolimits_{i < n} {\left\langle {\nabla _{{e_i}}^\partial {\nabla ^\partial }u,A{e_i}} \right\rangle }  + \sum\nolimits_{i < n} {\left\langle {I{I^\partial }({\nabla ^\partial }u,{e_i}),A{e_i}} \right\rangle }
\\&&+\sum\nolimits_{i < n} {\left\langle {\overrightarrow{n},A\left( {\left( {{e_i}.{u_n}} \right){e_i}} \right)} \right\rangle }  + \sum\nolimits_{i < n} {{u_n}\left\langle {\nabla _{{e_i}}^{}\overrightarrow{n},A{e_i}} \right\rangle }  + \left\langle {{\nabla _{\overrightarrow{n}}}\nabla u,A\overrightarrow{n}} \right\rangle
\\&=& L_A^\partial (u) + \sum\nolimits_{i < n} {\left\langle {\left\langle {sha{p^\partial }({\nabla ^\partial }u),{e_i}} \right\rangle \overrightarrow{n},A{e_i}} \right\rangle }  + \left\langle {\overrightarrow{n},A\left( {\sum\nolimits_{i < n} {\left( {{e_i}.{u_n}} \right){e_i}} } \right)} \right\rangle
\\&&-\sum\nolimits_{i < n} {{u_n}\left\langle {sha{p^\partial }\left( {{e_i}} \right),A{e_i}} \right\rangle }  + \left\langle {{\nabla _{\overrightarrow{n}}}\nabla u,A\overrightarrow{n}} \right\rangle
\\&=&L_A^\partial (u) + \left\langle {\overrightarrow{n},A\left( {\sum\nolimits_{i < n} {\left\langle {sha{p^\partial }({\nabla ^\partial }u),{e_i}} \right\rangle {e_i}} } \right)} \right\rangle  + \left\langle {A\overrightarrow{n},{\nabla ^\partial }{u_n}} \right\rangle
\\&&- {u_n}\sum\nolimits_{i < n} {\left\langle {A\,\circ \,sha{p^\partial }\left( {{e_i}} \right),{e_i}} \right\rangle }  + \left\langle {{\nabla _{\overrightarrow{n}}}\nabla u,A\overrightarrow{n}} \right\rangle
\\&=& L_A^\partial (u) + \left\langle {\overrightarrow{n},A\left( {sha{p^\partial }({\nabla ^\partial }u)} \right)} \right\rangle  + \left\langle {A\overrightarrow{n},{\nabla ^\partial }{u_n}} \right\rangle  - {u_n}H_A^\partial  + Hess(u)(\overrightarrow{n},A\overrightarrow{n}),
\end{eqnarray*}
where by definition we have $L_A^\partial u = \sum\nolimits_{i < n} {\left\langle {\nabla _{{e_i}}^\partial {\nabla ^\partial }u,A{e_i}} \right\rangle } $ and $ "shap" $ is the shape operator of $\partial M$ with respect to outward unit vector field $ n $ of $ M $. We also define  $H_A^\partial : = Trace(A\circ sha{p^\partial })$ as a generalization of the mean curvature of the boundary of $M$. \par
We know that (\ref{R4}) is obtained by integration from $\left\langle {{\nabla _{An}}\nabla u,\nabla u} \right\rangle  - {u_n}{L_A}u$. So we have,
\begin{eqnarray*}
\left\langle {{\nabla _{A\overrightarrow{n}}}\nabla u,\nabla u} \right\rangle  - {u_n}{L_A}u&=& Hess(u)(\nabla u,A\overrightarrow{n}) - {u_n}{L_A}u{\mkern 1mu} {\mkern 1mu} {\mkern 1mu}
\\&=& Hess(u)(\nabla u,A\overrightarrow{n}) - {u_n}L_A^\partial (u) - {u_n}\left\langle {A\overrightarrow{n},\left( {sha{p^\partial }({\nabla ^\partial }u)} \right)} \right\rangle
\\&&- {u_n}\left\langle {A\overrightarrow{n},{\nabla ^\partial }{u_n}} \right\rangle  + u_n^2H_A^\partial  - Hess(u)({u_n}\overrightarrow{n},A\overrightarrow{n}).
\end{eqnarray*}
But,
\begin{eqnarray*}
&&Hess(u)(\nabla u,A\overrightarrow{n}) - Hess(u)({u_n}\overrightarrow{n},A\overrightarrow{n})\\&&= Hess(u)(\nabla u - {u_n}\overrightarrow{n},A\overrightarrow{n})=Hess(u)({\nabla ^\partial }u,A\overrightarrow{n})= \left\langle {{\nabla _{{\nabla ^\partial }u}}\nabla u,A\overrightarrow{n}} \right\rangle\\&&=\left\langle {{\nabla _{{\nabla ^\partial }u}}{\nabla ^\partial }u,A\overrightarrow{n}} \right\rangle  + \left( {{\nabla ^\partial }u.{u_n}} \right)\left\langle {\overrightarrow{n},A\overrightarrow{n}} \right\rangle  - {u_n}\left\langle {sha{p^\partial }({\nabla ^\partial }u),A\overrightarrow{n}} \right\rangle.
\end{eqnarray*}
Hence the extended Reilly formula when $ A $ is parallel becomes as follows,
\[B = C,\]
where
\begin{eqnarray*}
B&=& \int_{\partial M} {\left( {\left\langle {{\nabla _{{\nabla ^\partial }u}}{\nabla ^\partial }u,A\overrightarrow{n}} \right\rangle  - 2{u_n}\left\langle {sha{p^\partial }({\nabla ^\partial }u),A\overrightarrow{n}} \right\rangle } \right)dvo{l_{\bar g}}}
\\&&+ \int_{\partial M} {\left( {u_n^2H_A^\partial  - {u_n}\left\langle {A\overrightarrow{n},{\nabla ^\partial }{u_n}} \right\rangle } \right)dvo{l_{\bar g}}}
\\&&+ \int_{\partial M} {\left( {\left( {{\nabla ^\partial }u.{u_n}} \right)\left\langle {\overrightarrow{n},A\overrightarrow{n}} \right\rangle  - {u_n}L_A^\partial (u)} \right)dvo{l_{\bar g}},}
\end{eqnarray*}
and
\[C = \int_M {Trace\left( {A\circ hes{s^2}\left( u \right)} \right)dvo{l_g}}  - \int_M {({L_A}u)(\Delta u)dvo{l_g}}  + \int_M {Ri{c_A}(\nabla u,\nabla u)dvo{l_g}}. \]
\end{proof}
\begin{remark} We define the second fundamental form of a hypersurface ${\Sigma ^{}} \subset {M^{}}$ by the rule $II (X,Y) = {\left( {{\nabla _X}Y} \right)^ \bot } =  - \left\langle {{\nabla _X}N,Y} \right\rangle $. So we call a hypersurface "convex" if the second fundamental form is negative definite.
\end{remark}
\begin{remark}Note that the result of the Theorem \ref{Reillyp} depends on the conditions ${\Delta _{{\nabla _{\nabla u}}A}}u = 0$ and $ div(A)=0 $, not parallelness of $ A $.
\end{remark}
\section{Estimates of the first eigenvalue of $ L_A $ when $ A $ is parallel}
By the Rielly-type formula, we get some estimates for the first positive eigenvalue of operator ${L_A}$ by some restrictions on $Ri{c_A}$, when $A$ is parallel and the manifold is compact. For manifolds with boundary, we denote ${\lambda _D}$, ${\lambda _N}$ for the first eigenvalue with the Dirichlet and Neumann boundary conditions respectively.

\begin{proof}[\textbf{Proof of the Theorem \ref{eigen1}}]
Since $A$ is parallel, one has,
\[{\delta _1}\left\langle {X,X} \right\rangle  \le \left\langle {AX,X} \right\rangle  \le {\delta _n}\left\langle {X,X} \right\rangle, \]
where ${\delta _1},{\delta _n}$ are defined as Definition \ref{minmax}.
Let ${\lambda}$ be the first positive eigenvalue of the operator ${L_A}(u) = div(A\nabla u)$ and $u$ be the corresponding eigenfunction, i.e., ${L_A}(u) + \lambda u = 0$. We can assume $u > 0$ and $\int_M {{u^2}dvo{l_g}}  = 1$. So we have,
\[{\delta _1}\int_M {{{\left| {\nabla u} \right|}^2}} dvo{l_g} \le \int_M {\left\langle {A\nabla u,\nabla u} \right\rangle dvo{l_g}}  =  - \int_M {u{L_A}(u)dvo{l_g}}  = \lambda \int_M {{u^2}dvo{l_g}}  = \lambda. \]
Similarly, one has ${\delta _n}\int_M {{{\left| {\nabla u} \right|}^2}} dvo{l_g} \ge \lambda $, so
\begin{equation}
\frac{\lambda }{{{\delta _n}}} \le \int_M {{{\left| {\nabla u} \right|}^2}} dvo{l_g} \le \frac{\lambda }{{{\delta _1}}}.
\label{e1}
\end{equation}
Since $M$ has empty boundary, by Theorem \ref{Reillyp} we get,
\[\int_M {Trace\left( {A \circ hes{s^2}\left( u \right)} \right)dvo{l_g}}  - \int_M {({L_A}u)(\Delta u)dvo{l_g}}  + \int_M {Ri{c_A}(\nabla u,\nabla u)dvo{l_g}}  = 0.\]
As ${L_A}u = - \lambda u$, we obtain,
\[ - \int_M {({L_A}u)(\Delta u)dvo{l_g}}  = \int_M {\lambda u(\Delta u)dvo{l_g}}  =  - \lambda \int_M {{{\left| {\nabla u} \right|}^2}dvo{l_g}}. \]
  Now, we prove the two estimates of the theorem.
\begin{itemize}
\item[1)] We note that $\int_M {Trace\left( {A\circ{{\left( {hess^2\left( u \right)} \right)}}} \right)dvo{l_g}}  \ge 0$, so
\[ - \lambda \int_M {{{\left| {\nabla u} \right|}^2}dvo{l_g}}  + \int_M {Ri{c_A}(\nabla u,\nabla u)dvo{l_g}}  \le 0,\]
and by assumption $Ri{c_A} \ge Kg$, hence $\left( {K - \lambda } \right)\int_M {{{\left| {\nabla u} \right|}^2}dvo{l_g}}  \le 0$ or equivalently $K \le \lambda $.
\\
To get a better estimate for $\lambda $, we use,
\begin{equation}
Trace\left( {A \circ hes{s^2}\left( u \right)} \right) \ge \frac{{{{({L_A}u)}^2}}}{{{\mkern 1mu} Trace\left( {A{\mkern 1mu} } \right)}},
\label{e111}
\end{equation}
By Theorem \ref{Reillyp} we see that (note $ Trace(A) $ is constant),
\[\frac{1}{{\,Trace\left( {A\,} \right)}}\int_M {{{({L_A}u)}^2}dvo{l_g}}  - \int_M {({L_A}u)(\Delta u)dvo{l_g}}  + \int_M {Ri{c_A}(\nabla u,\nabla u)dvo{l_g}}  \le 0.\]
So,
\[\frac{1}{{\,Trace\left( {A\,} \right)}}{\lambda ^2} + \left( {K - \lambda } \right)\int_M {{{\left| {\nabla u} \right|}^2}dvo{l_g}}  \le 0.\]
And by (\ref{e1}) we get,
\[\frac{{{\lambda ^2}}}{{\,Trace\left( {A\,} \right)}} + \left( {K - \lambda } \right)\frac{\lambda }{{{\delta _1}}} \le 0.\]
And finally,
\[\lambda  \ge \frac{{Trace\left( {A\,} \right)K}}{{Trace\left( {A\,} \right) - {\delta _1}}}.\]
\item[2)] By Theorem \ref{Reillyp} and the assumption on ${Ri{c_A}}$  we have,
\[\frac{1}{{\,Trace\left( {A\,} \right)}}\int_M {{{(\lambda u)}^2}dvo{l_g}}  + \int_M {(\lambda u)(\Delta u)dvo{l_g}}  + K\int_M {\left\langle {A\nabla u,\nabla u} \right\rangle dvo{l_g}}  \le 0.\]
We note that
\[\int_M {\left\langle {A\nabla u,\nabla u} \right\rangle dvo{l_g}}  =  - \int_M {u{L_A}(u)dvo{l_g}}  = \lambda \int_M {{u^2}dvo{l_g}}  = \lambda. \]
So
\[\frac{1}{{\,Trace\left( {A\,} \right)}}{\lambda ^2} - \lambda \int_M {{{\left| {\nabla u} \right|}^2}dvo{l_g}}  + K\lambda  \le 0.\]
And by (\ref{e1}) we get,
\[\lambda  \ge \frac{{Trace\left( {A\,} \right){\delta _1}K}}{{Trace\left( {A\,} \right) - {\delta _1}}}.\]
\end{itemize}
For the rigidity result,  let $ u $ be the corresponding eigenfunction, the equality in (\ref{e1b1}) and (\ref{e1b2}) implies equality in (\ref{e111}).  By proposition \ref{newton} and Cauchy-Schwartz inequality one has
\[A = \alpha I\,\,\,,\,\,\,Hessu = hI\,\,,\,\,\lambda  = \frac{{n\alpha K}}{{(n - 1)}},\]
where $\alpha ,h$ be smooth functions. So from ${L_A}u = {\Delta _A}u =  - \frac{{n\alpha Ku}}{{(n - 1)}}$ we conclude $nh\alpha  = {\Delta _A}u =  - \frac{{n\alpha Ku}}{{(n - 1)}}$, equivalently $h =  - \frac{K}{{(n - 1)}}u$ and the result follows by proposition \ref{hess}.
\end{proof}
\begin{remark} If $A = I{d_{TM}}$ we get the Lichnerowicz and Obata estimate for the first eigenvalue
in both cases.
\end{remark}
\begin{corollary} Let ${\Sigma ^n} \subset {M^{n + 1}}$ be a complete orientable Riemannian hypersurface or be an orientable space-like hypersurface in Lorentzian manifold $M$, which ${P_k}$ is parallel and positive-definite operator and
\begin{itemize}
\item[a)]$Ri{c_{{P_k}}} \ge \Lambda  > 0$,
\item[b)]$Ri{c_{{P_k}}}(X,X) \ge \Lambda {\left| X \right|^2}$, where $\Lambda  > 0$ and $X \in \Gamma \left( {TM} \right)$ is arbitrary vector field.
\end{itemize}
Then one has the following estimates of the first eigenvalue $ {\lambda _1}$ of the operator ${L_k}$,
\begin{itemize}
\item[a)]$\lambda  \ge \frac{{Trace(A)K}}{{Trace(A) - {\delta _1}}}\, \ge \frac{{{c_k}}}{{{c_k} - 1}}\Lambda,$
\item[b)]$\lambda  \ge \frac{{{c_k}{\delta _1}}}{{{c_k} - 1}}\Lambda $.
\end{itemize}
\end{corollary}
\begin{corollary}[\textbf{Dirichlet and Neumann boundary condition}]
Let $M$ be a compact Riemann manifold with boundary $\partial M$ and $A$ be a parallel symmetric and positive semi-definite operator on $M$ such that the outward unit vector field $\overrightarrow{n}$ is an eigenvector of $A$ and for Dirichlet boundary condition $\partial M$ be a convex hypersurface. Suppose that one of the following conditions holds,
\begin{itemize}
\item[1)]$Ri{c_A} \ge Kg$ and $K > 0$;
\item[2)]$Ri{c_A}(X,X) \ge K\left\langle {AX,X} \right\rangle$ and $K > 0$.
\end{itemize}
Then we have the following estimate for the first eigenvalue of the operator ${L_A}$ for Dirichlet or Neumann boundary condition,
\begin{itemize}
\item[1)]${\lambda _D},{\lambda _N} \ge \frac{{Trace\left( {A\,} \right)K}}{{Trace\left( {A\,} \right) - {\delta _1}}}\,$;
\item[2)]${\lambda _D},{\lambda _N} \ge \frac{{Trace\left( {A\,} \right){\delta _1}K}}{{Trace\left( {A\,} \right) - {\delta _1}}}$.
\end{itemize}
\end{corollary}
These estimates are not useful when $K  \le 0$. So we use the Li and Yau method \cite{li1980estimates} to estimate the first eigenvalue when $Ri{c_A} \le 0$. In principle Li-Yau's method gets a gradient estimate on the ${\lambda _1}$-eigenfunction by using the Bochner formula and maximal principle.\\

\begin{proof}[\textbf{Proof of the Theorem \ref{eigenpp2} }] We follow Theorem 5.3 p. 39 in \cite{li1980estimates} to get the estimate. In principal, we use generalized Bochner formula in Theorem \ref{extendedBochner} and maximum principle similar  \cite{li1980estimates}. Let $  u $ be an eigenfunction with ${\Delta _A}u = {L_A}u =  - \lambda u$, where $\lambda  > 0$. We can assume,
\[\int\limits_M u\, dvo{l_g} = 0,\,\,\,1 > \mathop {\max }\limits_{x \in M}u  > \mathop {\min }\limits_{x \in M} u =  - 1.\]
Let us consider the function $v = \ln (a + u)$ for some constant $a > 1$. The function $ v $ satisfies,
\[\nabla v = \frac{1}{{a + u}}\nabla u,\]
and
\[{\Delta _A}v = \frac{1}{{a + u}}{\Delta _A}u - \frac{1}{{{{(a + u)}^2}}}\left| {\nabla u} \right|_A^2 = \frac{{ - \lambda u}}{{a + u}} - \left| {\nabla v} \right|_A^2,\]
where $\left| X \right|_A^2 := \left\langle {X,AX} \right\rangle$. Define $Q\left( x \right): = \left| {\nabla v} \right|_{}^2\left( x \right)$, by the extended Bochner formula, one has,
\begin{eqnarray}\nonumber
\frac{1}{2}{\Delta _A}\left( Q \right)&=&\frac{1}{2} {\Delta _A}\left( {\left| {\nabla v} \right|_{}^2} \right) = Trace\left( {A \circ hes{s^2}v} \right) + \nabla v.{\Delta _A}v + Ri{c_A}\left( {\nabla v,\nabla v} \right)\\\label{eigenneg1}&\geq&\frac{1}{{Trace\left( A \right)}}{{\left( {{\Delta _A}v} \right)}^2} + \nabla v.{\Delta _A}v + Ri{c_A}\left( {\nabla v,\nabla v} \right).
\end{eqnarray}
But,
\begin{eqnarray*}
\frac{1}{{Trace\left( A \right)}}{\left( {{\Delta _A}v} \right)^2} &=& \frac{1}{{Trace\left( A \right)}}{\left( {\frac{{ - \lambda u}}{{a + u}} - \left| {\nabla v} \right|_A^2} \right)^2}\\& \ge &\frac{1}{{Trace\left( A \right)}}\left( {\left| {\nabla v} \right|_A^4 + \frac{{2\lambda u}}{{a + u}}\left| {\nabla v} \right|_A^2} \right),
\end{eqnarray*}
and
\begin{eqnarray}\nonumber
\nabla v.{\Delta _A}v&=& \nabla v.\left( {\frac{{ - \lambda u}}{{a + u}} - \left| {\nabla v} \right|_A^2} \right) =  - \lambda \frac{a}{{a + u}}\left| {\nabla v} \right|_{}^2 - \nabla v.\left| {\nabla v} \right|_A^2
\\\label{eigenneg2}&=&- \lambda \frac{a}{{a + u}}\left| {\nabla v} \right|_{}^2 - 2Hessv\left( {A\nabla v,\nabla v} \right) - \left\langle {\nabla v,\left( {{\nabla _{\nabla v}}A} \right)\nabla v} \right\rangle .
\end{eqnarray}
By parallelness of $ A $, we have
\begin{equation}
\nabla v.{\Delta _A}v =  - \lambda \frac{a}{{a + u}}\left| {\nabla v} \right|_{}^2 - \left( {A\nabla v} \right).{\left| {\nabla v} \right|^2}.
\label{eigenneg3}
\end{equation}
So (\ref{eigenneg1}) is written as
\begin{equation}
\frac{1}{2}{\Delta _A}\left( Q \right) + \left( {A\nabla v} \right).{\left| {\nabla v} \right|^2} \ge \frac{{\delta _1^2}}{{Trace\left( A \right)}}{Q^2} + \left[ {\lambda \left( {\frac{{2u}}{{\left( {a + u} \right)Trace\left( A \right)}} - \frac{a}{{a + u}}} \right) - K} \right]Q.
\label{eigenneg4}
\end{equation}
If ${x_0}$ is the maximum point, where $ Q $ achieves its maximum, then $\frac{1}{2}{\Delta _A}\left( Q \right) + \left( {A\nabla v} \right).{\left| {\nabla v} \right|^2} \le 0$ and we have,
\[Q\left( x \right) \le Q\left( {{x_0}} \right) \le \frac{{Trace\left( A \right)}}{{\delta _1^2}}\left[ {\frac{a}{{a - 1}}\lambda  + K} \right],\]
for all $x \in M$. Integrating $\sqrt Q  = \left| {\nabla \ln (u + a)} \right|$ along a minimal geodesic $\gamma $ joining points $u =  - 1$ to $u = \max u\geq0$, we have,
\[\ln \left( {\frac{a}{{a - 1}}} \right) \le \ln \left( {\frac{{a + \max u}}{{a - 1}}} \right) \le \int\limits_\gamma  {\left| {\nabla \ln (a + u)} \right|}  \le \frac{{d\sqrt {Trace\left( A \right)} }}{{{\delta _1}}}{\left[ {\frac{a}{{a - 1}}\lambda  + K} \right]^{\frac{1}{2}}}.\]
Setting $t = \frac{{a - 1}}{a}$, we have
\[t\left( {\frac{{\delta _1^2}}{{{d^2}Trace\left( A \right)}}{{\left( {\ln t} \right)}^2} - K} \right) \le \lambda .\]
By maximizing, we have,
\[2\left( {\alpha  + \sqrt {{\alpha ^2} + K\alpha } } \right)\exp \left( { - 1 - \sqrt {1 + \frac{K}{\alpha }} } \right) \le \lambda ,\]
where
$\alpha  = \frac{{\delta _1^2}}{{{d^2}Trace\left( A \right)}}$.

\end{proof}
\begin{remark}Similar method can be applied for the manifold with boundary and $Ri{c_A} \ge K$, where $K \le 0$.
\end{remark}

\section{Reilly formula, when $ A $ is Codazzi tensor}

\begin{proof}[\textbf{Proof of the Theorem \ref{Reilly}}] To get the Reilly formula for general case we should integrate of the extended Bochner formula (\ref{extendedBochner}). So we should add the phrase
\[\int_M {\frac{1}{2}\left( {\left\langle {\nabla {{\left| {\nabla u} \right|}^2},div(A)} \right\rangle  - {\Delta _{\left( {{\nabla _{\nabla u}}A} \right)}}u} \right)} \,dvo{l_g},\]
to the Reilly- type formula (\ref{RR1}).
By assumption,for any vector field $ X,Y $ one has $\left( {{\nabla _X}A} \right)Y = \left( {{\nabla _Y}A} \right)X$. Let $\left\{ {{e_i}} \right\}$ be a local orthonormal frame field with ${\nabla _{{e_i}}}{e_j}(x) = 0$. So at the point $ x $ one has
\begin{eqnarray*}
{\Delta _{\left( {{\nabla _{\nabla u}}A} \right)}}u&=&{L_{\left( {{\nabla _{\nabla u}}A} \right)}}u - \left\langle {\nabla u,div({\nabla _{\nabla u}}A)} \right\rangle
\\&=& {L_{\left( {{\nabla _{\nabla u}}A} \right)}}u - \left\langle {\nabla u,\sum\nolimits_i {{\nabla _{{e_i}}}\left( {({\nabla _{\nabla u}}A){e_i}} \right)} } \right\rangle
\\&=& {L_{\left( {{\nabla _{\nabla u}}A} \right)}}u - \left\langle {\nabla u,\sum\nolimits_i {{\nabla _{{e_i}}}\left( {({\nabla _{{e_i}}}A)\nabla u} \right)} } \right\rangle
\\&=&= {L_{\left( {{\nabla _{\nabla u}}A} \right)}}u - \sum\nolimits_i {\left\langle {\nabla u,(\nabla _{{e_i}}^2A)\nabla u} \right\rangle }  - \sum\nolimits_i {\left\langle {\nabla u,({\nabla _{{e_i}}}A){\nabla _{{e_i}}}\nabla u} \right\rangle }
\\&=& {L_{\left( {{\nabla _{\nabla u}}A} \right)}}u - \left\langle {\nabla u,\left( {\Delta A} \right)\nabla u} \right\rangle  - \sum\nolimits_i {\left\langle {({\nabla _{{e_i}}}A)\nabla u,{\nabla _{{e_i}}}\nabla u} \right\rangle }
\\&=&{L_{\left( {{\nabla _{\nabla u}}A} \right)}}u - \left\langle {\nabla u,\left( {\Delta A} \right)\nabla u} \right\rangle  - \sum\nolimits_i {\left\langle {({\nabla _{\nabla u}}A){e_i},{\nabla _{{e_i}}}\nabla u} \right\rangle }
\\&=&{L_{\left( {{\nabla _{\nabla u}}A} \right)}}u - \left\langle {\nabla u,\left( {\Delta A} \right)\nabla u} \right\rangle  - {\Delta _{\left( {{\nabla _{\nabla u}}A} \right)}}u
\end{eqnarray*}
Equivalently
\[{\Delta _{\left( {{\nabla _{\nabla u}}A} \right)}}u = \frac{1}{2}\left( {{L_{\left( {{\nabla _{\nabla u}}A} \right)}}u - \left\langle {\nabla u,\left( {\Delta A} \right)\nabla u} \right\rangle } \right),\]
so
\begin{eqnarray*}
- \int_M {\left( {{\Delta _{\left( {{\nabla _{\nabla u}}A} \right)}}u} \right)} dvo{l_g}&=& - \frac{1}{2}\int_M {\left( {{L_{\left( {{\nabla _{\nabla u}}A} \right)}}u - \left\langle {\nabla u,\left( {\Delta A} \right)\nabla u} \right\rangle } \right)} dvo{l_g}
\\&=&\frac{1}{2}\int_M {\left\langle {\nabla u,\left( {\Delta A} \right)\nabla u} \right\rangle dvo{l_g}}  - \frac{1}{2}\int_M {\left( {{L_{\left( {{\nabla _{\nabla u}}A} \right)}}u} \right)dvo{l_g}}.
\end{eqnarray*}
But $\nabla u = {\nabla ^\partial }u + {u_n}\overrightarrow{n}$, then
\begin{eqnarray*}
\int_M {\left( {{L_{\left( {{\nabla _{\nabla u}}A} \right)}}u} \right)dvo{l_g}}&=&\int_{\partial M} {\left\langle {\left( {{\nabla _{\nabla u}}A} \right)\nabla u,\overrightarrow{n}} \right\rangle dvo{l_{\bar g}}}  = \int_{\partial M} {\left\langle {\nabla u,\left( {{\nabla _{\nabla u}}A} \right)\overrightarrow{n}} \right\rangle dvo{l_{\bar g}}}
\\&=& \int_{\partial M} {\left\langle {\nabla u,\left( {{\nabla _{\overrightarrow{n}}}A} \right)\nabla u} \right\rangle dvo{l_{\bar g}}}
\\&=&\int_{\partial M} {\left\langle {{\nabla ^\partial }u,\left( {{\nabla _{\overrightarrow{n}}}A} \right){\nabla ^\partial }u} \right\rangle dvo{l_{\bar g}}} \\&& + 2\int_{\partial M} {{u_n}\left\langle {{\nabla ^\partial }u,\left( {{\nabla _{\overrightarrow{n}}}A} \right)\overrightarrow{n}} \right\rangle dvo{l_{\bar g}}}
\\&&+ \int_{\partial M} {u_n^2\left\langle {\overrightarrow{n},\left( {{\nabla _{\overrightarrow{n}}}A} \right)\overrightarrow{n}} \right\rangle dvo{l_{\bar g}}}.
\end{eqnarray*}
Also one knows
\begin{eqnarray*}
 &&\int_M {\frac{1}{2}\left\langle {\nabla {{\left| {\nabla u} \right|}^2},div(A)} \right\rangle } \,dvo{l_g} \\&& =\frac{1}{2}\int_M {div\left( {{{\left| {\nabla u} \right|}^2}divA} \right)} dvo{l_g} - \frac{1}{2}\int_M {{{\left| {\nabla u} \right|}^2}div\left( {divA} \right)} dvo{l_g} \\ &&= \frac{1}{2}\int_{\partial M} {{{\left| {\nabla u} \right|}^2}\left\langle {divA,\overrightarrow{n}} \right\rangle dvo{l_{\bar g}}}  - \frac{1}{2}\int_M {{{\left| {\nabla u} \right|}^2}div\left( {divA} \right)} dvo{l_g},
\end{eqnarray*}
and the formula follows.
\end{proof}

By easy computation, one knows ${\rm{divA}} = \nabla Trace(A)$, so when $ A $ is divergence free Codazzi tensor, then $Trace(A)$ is constant. In this case, we have the following results for the estimate of the first eigenvalue.
\begin{proof}[\textbf{Proof of the Theorem \ref{eigennp1}}]
By assumption one has,
\[{\delta _1}\left\langle {X,X} \right\rangle  \le \left\langle {AX,X} \right\rangle  \le {\delta _n}\left\langle {X,X} \right\rangle, \]
where ${\delta _1},{\delta _n}$ are defined as Definition \ref{minmax}.
Let ${\lambda}$ be a positive eigenvalue of the operator ${L_A}(u) = div(A\nabla u)$ and $u$ be the corresponding eigenfunction, i.e., ${L_A}(u) + \lambda u = 0$. We can assume $u > 0$ and $\int_M {{u^2}dvo{l_g}}  = 1$. So we have,
\[{\delta _1}\int_M {{{\left| {\nabla u} \right|}^2}} dvo{l_g} \le \int_M {\left\langle {A\nabla u,\nabla u} \right\rangle dvo{l_g}}  =  - \int_M {u{L_A}(u)dvo{l_g}}  = \lambda \int_M {{u^2}dvo{l_g}}  = \lambda. \]
Similarly, one has ${\delta _n}\int_M {{{\left| {\nabla u} \right|}^2}} dvo{l_g} \ge \lambda $, so
\begin{equation}
\frac{\lambda }{{{\delta _n}}} \le \int_M {{{\left| {\nabla u} \right|}^2}} dvo{l_g} \le \frac{\lambda }{{{\delta _1}}}.
\label{eigennonp1}
\end{equation}
Since $M$ has empty boundary, by the extended Reilly formula  we get,
\begin{eqnarray}\nonumber
0&=&\int_M {\left( {Trace\left( {A \circ hes{s^2}(u)} \right)} \right)} dvo{l_g} - \int_M {{\Delta _A}(u)\left( {\Delta u} \right)} dvo{l_g}
\\\label{eigennonp2}&&
 + \int_M {Ri{c_A}\left( {\nabla u,\nabla u} \right)} dvo{l_g}
+\frac{1}{2}\int_M {\left\langle {\nabla u,\left( {\Delta A} \right)\nabla u} \right\rangle dvo{l_g}}
\\\nonumber&& - \frac{1}{2}\int_M {{{\left| {\nabla u} \right|}^2}div(divA)dvo{l_g}}.
\end{eqnarray}
By assumption $ Trace(A) $ is constant and $ T^A=0 $ so $divA = 0$. Also by Lemma \ref{laplaceA} we have
\[\left\langle {\nabla u,\left( {\Delta A} \right)\nabla u} \right\rangle  =  - Ri{c_A}(\nabla u,\nabla u) + Ric(\nabla u,A\nabla u),\]
Hence the  equation (\ref{eigennonp2}) is written as
\begin{eqnarray*}
0&=&\int_M {\left( {Trace\left( {A \circ hes{s^2}(u)} \right)} \right)} dvo{l_g} - \int_M {{\Delta _A}(u)\left( {\Delta u} \right)} dvo{l_g} \\&&+ \frac{1}{2}\int_M {Ri{c_A}\left( {\nabla u,\nabla u} \right)} dvo{l_g}
+ \frac{1}{2}\int_M {Ric(\nabla u,A\nabla u)dvo{l_g}}.
\end{eqnarray*}
By $ div(A)=0 $ we have ${L_A}u = {\Delta _A}u =  - \lambda u$. So by the restriction on $Ri{c_A}(X,X) + Ric(X,AX)$ we have
\[\frac{1}{{\,Trace\left( {A\,} \right)}}{\lambda ^2} + \left( {K - \lambda } \right)\int_M {{{\left| {\nabla u} \right|}^2}dvo{l_g}}  \le 0.\]
As a similar discussion in Theorem \ref{eigen1} and (\ref{eigennonp1}) we get,
\[\frac{{{\lambda ^2}}}{{\,Trace\left( {A\,} \right)}} + \left( {K - \lambda } \right)\frac{\lambda }{{{\delta _1}}} \le 0.\]
And finally,
\[\lambda  \ge \frac{{Trace\left( {A\,} \right)K}}{{Trace\left( {A\,} \right) - {\delta _1}}}.\]
The equality condition is as Theorem \ref{eigen1}.
\end{proof}

\begin{proof}[\textbf{Proof of the Theorem \ref{eigennp2}}] By proposition \ref{lcomplicated}  and lemma \ref{laplaceA10}    we have
\begin{eqnarray*}
{\Delta _{\left( {{\nabla _{{\nabla _u}}}A} \right)}}u &=& Ri{c_A}\left( {\nabla u,\nabla u} \right) - Ric\left( {\nabla u,A\nabla u} \right) + \sum\nolimits_i \left\langle {{T^{\left( {{\nabla _{\nabla u}}A} \right)}}({e_i},\nabla u),{e_i}} \right\rangle  \\&&
+ \sum\nolimits_i {{e_i}.\left\langle {\nabla u,{T^A}({e_i},\nabla u)} \right\rangle .}
\end{eqnarray*}
The tensor $ A $ is Codazzi, so $\sum\nolimits_i {{e_i}.\left\langle {\nabla u,{T^A}(\nabla u,{e_i})} \right\rangle  = 0}$ and $ div(A)=0 $, Thus the extended Bochner formula is became to
\begin{eqnarray*}
\frac{1}{2}{L_A}({\left| {\nabla u} \right|^2})&=& Trace\left( {A \circ hes{s^2}\left( u \right)} \right) + \left\langle {\nabla u,\nabla ({\Delta _A}u)} \right\rangle  + \sum\nolimits_i {\left\langle {{T^{\left( {{\nabla _{\nabla u}}A} \right)}}(\nabla u,{e_i}),{e_i}} \right\rangle } \\&& + Ric\left( {\nabla u,A\nabla u} \right).
\end{eqnarray*}
By the restrictions on $Ric\left( {\nabla u,A\nabla u} \right)$ and second covariant derivative of $ A $, we have
\[\frac{1}{2}{L_A}({\left| {\nabla u} \right|^2}) \geq Trace\left( {A \circ hes{s^2}\left( u \right)} \right) + \left\langle {\nabla u,\nabla ({\Delta _A}u)} \right\rangle  - \left( {K + 2K'} \right){\left| {\nabla u} \right|^2}.\]
As similar computation in the proof of theorem \ref{eigenpp2}, let $v: = \ln (a + u)$ and $Q: = {\left| {\nabla v} \right|^2}$, then we have,
\[\nabla v.{\Delta _A}v = \nabla v.\left( {\frac{{ - \lambda u}}{{a + u}} - \left| {\nabla v} \right|_A^2} \right) =  - \lambda \frac{a}{{a + u}}\left| {\nabla v} \right|_{}^2 - \nabla v.\left| {\nabla v} \right|_A^2.\]
Thus,
\begin{equation}
\frac{1}{2}{\Delta _A}Q \ge \frac{{\delta _1^2{Q^2}}}{{Trace\left( A \right)}} + \left( {\lambda \left( {\frac{{2 {\delta _1}u}}{{Trace\left( A \right)\left( {a + u} \right)}} - \frac{a}{{a + u}}} \right) - \left( {K + 2K'} \right)} \right)Q - \nabla v.\left| {\nabla v} \right|_A^2.
\label{L1}
\end{equation}

Now we need to estimate $\nabla v.\left| {\nabla v} \right|_A^2$.
Since
\begin{equation*}
\nabla v.\left| {\nabla v} \right|_A^2=2Hess(A\nabla v,\nabla v)+\langle \nabla v, (\nabla_{\nabla v})\nabla v\rangle\leq A\nabla v.Q+\delta Q
\end{equation*}
we can get
\begin{eqnarray*}
&&\frac{1}{2}{\Delta _A}\left( Q \right) +A\nabla v.Q\\&&\ge\frac{{\delta _1^2{Q^2}}}{{Trace\left( A \right)}} + \left( {\lambda ( {\frac{{2{\delta _1}u}}{{Trace\left( A \right){{\left( {a + u} \right)}}}} - \frac{a}{{a + u}}} ) - \left( {K + 2K'+\delta} \right)} \right)Q.
\end{eqnarray*}
If ${x_0}$ is the maximum point, where $ Q $ achieves its maximum, then ${\Delta _A}Q({x_0}) \le 0$ and\linebreak $\nabla Q({x_0}) = 0$, so
 \[Q \le Q({x_0}) \le \frac{{Trace\left( A \right)}}{{\delta _1^2}}\left( {\frac{{a\lambda }}{{a - 1}} +  {K + 2K'} +\delta} \right).\]
for all $x \in M$. Integrating $\sqrt Q  = \left| {\nabla \ln (u + a)} \right|$ along a minimal geodesic $\gamma $ joining points $u =  - 1$ to $u = \max u\geq0$, we have,
\[\ln \left( {\frac{a}{{a - 1}}} \right) \le \ln \left( {\frac{{a + \max u}}{{a - 1}}} \right) \le \int\limits_\gamma  {\left| {\nabla \ln (a + u)} \right|}  \le \frac{{d\sqrt {Trace\left( A \right)} }}{{{\delta _1}}}{\left[ {\frac{a}{{a - 1}}\lambda  + K + 2K'+\delta} \right]^{\frac{1}{2}}}.\]
Setting $t = \frac{{a - 1}}{a}$, we have
\[{\left( {\ln (\frac{1}{t})} \right)^2} \le \frac{{{d^2}Trace(A)}}{{\delta _1^2}}\left[ {\frac{\lambda }{t} + K + 2K'+\delta} \right],\]
in other words,
\[t\left( {\frac{{\delta _1^2}}{{{d^2}Trace(A)}}{{\left( {\ln (t)} \right)}^2} - \left(  K + 2K'+\delta \right)} \right) \le \lambda .\]
Consequently, by maximizing, we have,
\[2\left( {\alpha  + \sqrt {{\alpha ^2} + ( K + 2K'+\delta)\alpha } } \right)\exp \left( { - 1 - \sqrt {1 + \frac{{ K + 2K'+\delta}}{\alpha }} } \right) \le \lambda ,\]
where
$\alpha  = \frac{{\delta _1^2}}{{{d^2}Trace\left( A \right)}}.$
\end{proof}

\section{estimate of the first eigenvalue in general case}
Let $ B $ be a (1,1)-self-adjoint tensor field on a closed manifold $ M $. We get an estimate for the first eigenvalue of the operator ${L_B}$. \\

\begin{proof}[\textbf{Proof of Theorem \ref{more general}}] From the extended Bochner formula, we know that
\begin{eqnarray*}
\frac{1}{2}{L_B}({\left| {\nabla u} \right|^2}) &=& \frac{1}{2}\left\langle {\nabla {{\left| {\nabla u} \right|}^2},div(B)} \right\rangle  + Trace\left( {B \circ hes{s^2}\left( u \right)} \right) + \left\langle {\nabla u,\nabla ({\Delta _B}u)} \right\rangle  \\&&- {\Delta _{\left( {{\nabla _{\nabla u}}B} \right)}}u + Ri{c_B}(\nabla u,\nabla u).
\end{eqnarray*}
Also,
\[\left\langle {\left( {\Delta B} \right)\nabla u,\nabla u} \right\rangle  = \left\langle {{\nabla _{\nabla u}}div(B),\nabla u} \right\rangle  - Ri{c_B}(\nabla u,\nabla u) + Ric(\nabla u,B\nabla u),\]
and
\begin{eqnarray*}
{\Delta _{\left( {{\nabla _{{\nabla _u}}}B} \right)}}u &=& \nabla u.\nabla u.Trace(B) - \left\langle {\nabla u,\left( {\Delta B} \right)\nabla u} \right\rangle  + \sum\nolimits_i {\left\langle {{T^{\left( {{\nabla _{\nabla u}}B} \right)}}(\nabla u,{e_i}),{e_i}} \right\rangle }  \\&&+ \sum\nolimits_i {{e_i}.\left\langle {\nabla u,{T^B}({e_i},\nabla u)} \right\rangle }.
\end{eqnarray*}
So from the conditions of the theorem, we have,
\begin{eqnarray*}
\frac{1}{2}{\Delta _B}({{\left| {\nabla u} \right|}^2}) &=&Trace\left( {B \circ hes{s^2}\left( u \right)} \right) + \left\langle {\nabla u,\nabla ({\Delta _B}u)} \right\rangle  - \sum\nolimits_i {\left\langle {{T^{\left( {{\nabla _{\nabla u}}B} \right)}}({e_i},\nabla u),{e_i}} \right\rangle }
\\&&- \sum\nolimits_i {{e_i}.\left\langle {\nabla u,{T^B}({e_i},\nabla u)} \right\rangle  + Ric(\nabla u,A\nabla u)} .
\end{eqnarray*}
Items (b) and (c) imply,
\[\frac{1}{2}{\Delta _B}({\left| {\nabla u} \right|^2}) \ge \frac{{{{\left( {{\Delta _B}u} \right)}^2}}}{{n{\delta _n}}} + \left\langle {\nabla u,\nabla ({\Delta _B}u)} \right\rangle  - \sum\nolimits_i {{e_i}.\left\langle {\nabla u,{T^B}({e_i},\nabla u)} \right\rangle }  + (K + 2nK'){\left| {\nabla u} \right|^2}.\]
We know ${L_B}u = {\Delta _B}u =  - \lambda u$ and $\partial M = \emptyset$, so by integration we have,
\begin{eqnarray*}
\int_M {\frac{1}{2}{\Delta _A}({{\left| {\nabla u} \right|}^2})dvo{l_g}}  &\ge& \frac{1}{{n{\delta _n}}}\int_M {{{\left( {{\Delta _A}u} \right)}^2}dvo{l_g}}  + \int_M {\left\langle {\nabla u,\nabla ({\Delta _A}u)} \right\rangle dvo{l_g}}
\\ && - \int_M {\sum\nolimits_i {{e_i}.\left\langle {\nabla u,{T^B}({e_i},\nabla u)} \right\rangle } dvo{l_g}} \\&& + \int_M {(K + 2nK'){{\left| {\nabla u} \right|}^2}dvo{l_g}} .
\end{eqnarray*}
But ${\sum\nolimits_i {{e_i}.\left\langle {\nabla u,{T^B}({e_i},\nabla u)} \right\rangle } }$ is of divergence form, so
$\int_M {\sum\nolimits_i {{e_i}.\left\langle {\nabla u,{T^B}({e_i},\nabla u)} \right\rangle } dvo{l_g}}  = 0$, also
$\frac{\lambda }{{{\delta _n}}} \le \int_M {{{\left| {\nabla u} \right|}^2}dvo{l_g}}  \le \frac{\lambda }{{{\delta _1}}}$, so we have the following inequality,
\[0 \ge \frac{{{\lambda ^2}}}{{n{\delta _n}}} - \frac{{{\lambda ^2}}}{{{\delta _1}}} + (K + 2nK')\frac{\lambda }{{{\delta _n}}}.\]
Consequently,
\[\lambda  \ge \frac{{n{\delta _n}\left( {K + 2nK'} \right)}}{{n{\delta _n} - {\delta _1}}}.\]
\end{proof}

{}


\begin{thebibliography}{00}
\bibitem{alencar2015eigenvalue} H. Alencar,  G. S. Neto, D. Zhou,  Eigenvalue estimates for a class of elliptic differential operators on compact
manifolds, Bulletin of the Brazilian Mathematical Society, New Series, 46(3)( 2015), 491-504.
\bibitem{alias2016maximum} L. J. Alias, P.  Mastrolia, M. Rigoli, et 'al.,  Maximum principles and geometric applications,  Springer;
2016.
\bibitem{chavel1984eigenvalues}I. Chavel,   Eigenvalues in Riemannian geometry, Academic press, 1984.
\bibitem{chen1997poincare}R.  Chen, P.  Li, On Poincare type inequalities, Transactions of the American Mathematical Society.
349(4)(1997), 1561-1585.
\bibitem{daicomparison} X. Dai, G. Wei, Comparison geometry for Ricci curvature,  http://math.ucsb.edu/ dai/Ricci-book.pdf, 2012.
\bibitem{do2010inequalities}M. P. Do Carmo, Q. Wang, C.  Xia,  Inequalities for eigenvalues of elliptic operators in divergence form on
Riemannian manifolds. Annali di Matematica Pura ed Applicata, 189(4) (2010), 643-660.

\bibitem{gomes2016eigenvalue}J. N. Gomes, J. F.  Miranda,  Eigenvalue estimates for a class of elliptic differential operators in divergence form,
Nonlinear Analysis, 176(2018), 1-19.
\bibitem{grigor2006heat}A.  Grigoryan,   Heat kernels on weighted manifolds and applications,  Cont Math,  398 (2006), 93-191.
\bibitem{li1980estimates}P.   Li,  Lecture note on geometric analysis, 1996.
\bibitem{li2012geometric} P.  Li,  Geometric analysis, Cambridge Studies in Advanced Mathematics, Cambridge University Press, Cambridge, 2012.

\bibitem{ma2009extension}L. Ma and S. H. Du, Extension of Reilly formula with applications to eigenvalue estimates for drifting Laplacians, C. R. Math. Acad. Sci. Paris., 348 (2010), 1203-1206.
\bibitem{reilly1977applications}R. C. Reilly,  Applications of the Hessian operator in a Riemannian manifold,  Indiana University Mathematics
Journal, 26(3)(1977), 459-472.
\bibitem{shi2016eigenvalue}J. Shi,  Eigenvalue estimates for some natural elliptic operators on hypersurfaces. Differential Geometry and
its Applications, 49( 2016), 97-110.


\end{thebibliography}
\end{document}